\documentclass[tbtags]{article}
\usepackage{amsmath,amsfonts,amsbsy,amsthm,latexsym,amssymb,enumerate,graphicx, psfrag,caption}

\newtheorem{theorem}{Theorem}
\newtheorem{lemma}[theorem]{Lemma}

\begin{document}
\sloppy

\title{The total external branch length of Beta-coalescents\footnote{Work partially supported by the DFG Priority Programme SPP 1590 ``Probabilistic Structures in Evolution''.  I. S. is supported by the  German Academic Exchange Service (DAAD).}}
 
\author{ G\"otz Kersting%
\thanks{Institut f\"ur Mathematik, Goethe-Universit\"at, 60054 Frankfurt am Main, Germany. \newline\texttt{kersting@math.uni-frankfurt.de, stanciu@math.uni-frankfurt.de, wakolbinger@math.uni-frankfurt.de}
} \qquad
 Iulia Stanciu$^\dagger$%
 \qquad
Anton Wakolbinger$^\dagger$%
}
  \maketitle
\begin{abstract}
For $1<\alpha <2$ we derive the  asymptotic distribution of the total length of {\em external} branches of a Beta$(2-\alpha, \alpha)$-coalescent as the number $n$ of leaves becomes large. It turns out  the fluctuations of the external branch length follow those of $\tau_n^{2-\alpha}$ over the entire parameter regime, where $\tau_n$ denotes the random number of coalescences that bring the $n$ lineages down to one. This is in contrast to the fluctuation behavior of the total branch length, which exhibits a transition at $\alpha_0 = (1+\sqrt 5)/2$ (\cite{ke}).
\end{abstract}
\begin{small}
\end{small}

\newpage

\section{Introduction and main results}\label{Intro}

The family of Beta coalescents belongs to the so-called $\Lambda$-coalescents introduced by Pitman \cite{pi} and Sagitov \cite{sa}, see the survey of N. Berestycki \cite{nber}. These are characterized by a probability measure $\Lambda(dp)$ on $[0,1]$. For $\Lambda = \delta_0$ one recovers the classical Kingman coalescent. For $\Lambda$ having no mass in $0$ they can be thought as modelling a random gene tree within a species in which single reproduction events affect a non-vanishing fraction of the population (see \cite{bobo,elwa,wa} for applications to certain maritime species). The basic ingredient for the coalescent dynamics is a Poisson process on $\mathbb R \times [0,1]$ with intensity measure  $dt \, \nu(dp) dt = dt\, p^{-2}  \Lambda(dp) $. With this Poisson process as a random input, the $(\Lambda,n)$-coalescent arises as follows: Initially there are $n$ lineages (or ``particles''), and whenever a Poisson point $(t,p)$ arrives, each of the lineages that exist at time $t$ independently takes part in the corresponding coalescence event with probability $p$. By assumption, the pair coalescence rate equals 1, so the events that are relevant for coalescences within the sample arrive at finite rate.   (More formally, this can be viewed as a partition-valued Markov process, see below.) 

For the choice $\nu(dp) = {\rm const} (p/(1-p))^\alpha$, $0<\alpha < 2$, the probability measure $\Lambda$ is the Beta($2-\alpha, \alpha$)-distribution.  For $\alpha=1$ one obtains the Bolthausen-Sznitman coalescent \cite{bo} and in the limit $\alpha \to 2$ one retrieves Kingman's coalescent. Beta coalescents play a prominent role because of their intimate connections to $\alpha$-stable continuum branching processes (see \cite{bi}) and  also figure as prominent alternatives to Kingman's coalescent from a statistical point of view, see \cite{bibla1, bibla2}. 

Recently there has been considerable progress in investigating the asymptotic distribution (as the sample size $n \to \infty$) of interesting functionals of Beta coalescents, such as  the total tree length $L_n$ or the  total length $\ell_n$ of all {\em external branches}, i.e. those branches which end in a leaf. These quantities are also of interest in population genetics, since in the infinite sites model $L_n$ figures as the (random) Poisson parameter of the total number of segregating sites in the sample, whereas $\ell_n$ is the (random) Poisson parameter of the total number of mutations carried by single individuals.

In this paper we will investigate the asymptotic distribution of the suitably normalized total external branch length $\ell_n$ in a (Beta($2-\alpha,\alpha$), $n$)-coalescent, with $1< \alpha < 2$. In order to state our main result and to put it into context with previous research, we give a few more formal definitions.


With a labelling of the initial particles by the numbers $1, \ldots,n$, the merging process is described by a sequence   $\Pi_0, \ldots, \Pi_{\tau_n}$ of partitions of $\{1, \ldots, n \}$. Here $\Pi_k$ consists of the classes of a (random) equivalence relation $\sim_k$ on $\{ 1, \ldots,n\}$, where $i \sim_k j$ states that $i,j \in \{1, \ldots, n\}$ have coalesced into one particle after $k$ merging events. In particular, $\Pi_0=\{ \{1\}, \ldots, \{n\}\}$ and $\Pi_{\tau_n}= \{\{1,\ldots,n\}\}$, where $\tau_n$ denotes the (random) total number of merging events.

The process proceeds in continuous time. At times 
$0= T_0 < T_1 < \cdots < T_{\tau_n}   $ particles merge, 
and $n=X_0 > X_1 > \cdots > X_{\tau_n}=1$ are the corresponding numbers of particles, thus 
\[X_k= \# \Pi_k\ . \] 
Note that the times $T_k$ and numbers $X_k$ depend on $n$ (which we suppress in the notation). For convenience we put $X_k=1$ for $k > \tau_n$. The total branch length of the tree is given by
\begin{align*} L_n = \sum_{k=0}^{\tau_n-1} X_{k}(T_{k+1}-T_{k})  \ .
\end{align*}
whereas the {\em total length $\ell_n$ of all external branches} can be written as
\begin{align} 
\ell_n = \sum_{k=0}^{\tau_n-1} Y_{k}(T_{k+1}-T_{k})  \ ,
\label{exlength}
\end{align}
with
\[ Y_k =  \# \{ i \le n: \{i\} \in \Pi_k \} \ , \]
which is the number of external branches, still present up to time $T_k$. Note that $Y_{\tau_n}=0$. For definiteness let $Y_k=0$ for $k>\tau_n$. 

The asymptotic distribution of $L_n$ and $\ell_n$ has been studied in various publications. Drmota et al \cite{dr} and Iksanov and M\"ohle \cite{ikmo} studied the total length of the Bolthausen-Sznitman coalescent. M\"ohle's investigation \cite{mo} covers the case $0<\alpha < 1$. In that case there is no substantial difference in the asymptotic behavior of $L_n$ and $\ell_n$, more precisely
\[ \frac {L_n} n \ \stackrel d\to \ S \quad \text{and}\quad \frac {\ell_n} n \ \stackrel d\to \ S  \]
with
\[ S= \int_0^\infty \exp(-\xi_t) \, dt \ , \]
where $(\xi_t)_{t \ge 0}$ denotes a certain driftless subordinator, depending on $\alpha$.

In case $1< \alpha < 2$ the situation is notably different. For this case Berestycki et al \cite{bebesw2,bebesw} (see the Theorems 9 and 1.9 therein) give the following results:
\[ \frac {L_n}{n^{2-\alpha}} \ \to \ \frac{\alpha(\alpha-1)\Gamma(\alpha)}{2-\alpha} \quad \text{and} \quad \frac {\ell_n}{L_n} \ \to \ 2-\alpha \]
in probability, see also Dhersin and Yuan \cite{deyu}. We shall see that the difference to the case $0<\alpha < 1$ become even more visible on the level of fluctuations.

Let $\varsigma$ denote a real-valued stable random variable with index  $ 1<\alpha<2 $, which is normalized by the properties
\begin{align} 
\mathbf E(\varsigma)=0 \ , \quad \mathbf P( \varsigma > x)=o( x^{-\alpha})\ , \quad \mathbf P( \varsigma<-x)\sim x^{-\alpha}  \label{stable}
\end{align}
for $x \to \infty$. Thus $\varsigma$ is maximally skewed among the stable distributions of index~$\alpha$. 

\newpage

Also let
\[ c_1 = \alpha(\alpha-1)\Gamma(\alpha) \ , \quad c_2 = \frac{\alpha(2-\alpha)(\alpha-1)^{\frac 1\alpha +1}\Gamma(\alpha)}{\Gamma(2-\alpha)^{\frac 1\alpha}} \ . \]

\begin{theorem}
\label{mainresult}
For the Beta-coalescent with $1< \alpha < 2$
\begin{align}\label{extfluc}
\frac {\ell_n - c_1 n^{2-\alpha}}{n^{\frac 1\alpha +1-\alpha}}\ \stackrel d\to\ c_2 \varsigma \ .
\end{align} 
\end{theorem}

\noindent
The corresponding statement for the Kingman coalescent can be found in Janson and Kersting \cite{jake}: up to a logarithmic correction in scale the analogue of \eqref{extfluc} also holds for the case $\alpha = 2$, then with a normal limiting distribution. 

It is of interest to contrast Theorem 1 with the corresponding statement for the total length $L_n$, obtained in Kersting \cite{ke}. For convenience we reproduce it here.

\begin{theorem}
\label{mainresult2}
Let
$ c_i' = \frac{c_i}{2-\alpha}$, $ i=1,2$. 
Then for $1< \alpha < 2${\em:}
\begin{enumerate}[{\em(i)}]
\item
If $1< \alpha < \tfrac 12 (1+ \sqrt 5)$ {\em (}thus $1+\alpha -\alpha^2>0${\em )}, then
\[ \frac  {L_n-  c_1'n^{2-\alpha}}{n^{\frac 1\alpha + 1 -\alpha}} \ \stackrel{d}{\to}  \ \frac{c_2'\varsigma}{(1+\alpha -\alpha^2)^{\frac 1\alpha}} \ . \]
\item
If $\alpha=\tfrac 12 (1+ \sqrt 5)$, then
\[ \frac{L_n-  c_1'n^{2-\alpha}}{(\log n)^{\frac 1\alpha} } \ \stackrel{d}{\to}\ c_2'\varsigma \ . \]
\item
If $\tfrac 12 (1+ \sqrt 5) < \alpha < 2$, then 
\[L_n-  c_1'n^{2-\alpha}\ \stackrel{d}{\to}\ \eta \ , \]
where $\eta$ is a non-degenerate random variable.
\end{enumerate}
\end{theorem}

\noindent
The contrast between the two theorems is striking. The different regimes in Theorem \ref{mainresult2}  and the transition at the golden ratio $\alpha_0=\tfrac 12 (1+ \sqrt 5)$  is no longer present in Theorem \ref{mainresult}. Instead we observe that $|\ell_n-c_1n^{2-\alpha}|$ converges to 0 or $\infty$, depending on whether $\alpha > \alpha_0  $ or $\alpha < \alpha_0$. This can be understood as follows: There are two sources of randomness at work in $L_n$. On the one hand fluctuations arise from the randomness within the 'topology' of the tree, more precisely from the random variables  $X_0 > X_1 > \cdots >X_{\tau_n}$. These fluctuations are of order $n^{\frac 1\alpha + 1 - \alpha}$. On the other hand also the waiting times $T_{k+1}-T_k$ contribute to the random fluctuations. Most substantial are those fluctuations arising close to the root of the tree, that is the times $T_{k+1}-T_k$ with $k$ close to $\tau_n$. These fluctuations are of order 1. Now either of these fluctuations may dominate  the other, depending on the sign of $\frac 1\alpha +1-\alpha$. This is reflected in Theorem 2. 

On the other hand the just mentioned  fluctuations of order one do not show up asymptotically in $\ell_n$. This is simply due to the fact that branches close to the root of the tree typically are {\em internal} branches. 

The fact  that for $\alpha > \alpha_0=1.61$ the external length $\ell_n$ is decoupled from the total length $L_n$ is illustrated by the plots in Figure \ref{Fig1}.\\\\\\
\psfrag{l}{$\ell_n$}
\psfrag{L}{$L_n$}
\psfrag{k}{$K_n^{2-\alpha}$}
\psfrag{a}{$\alpha=1.2$}
\psfrag{b}{$\alpha=1.5$}
\psfrag{c}{$\alpha=1.8$}
\includegraphics[width=11.5cm]{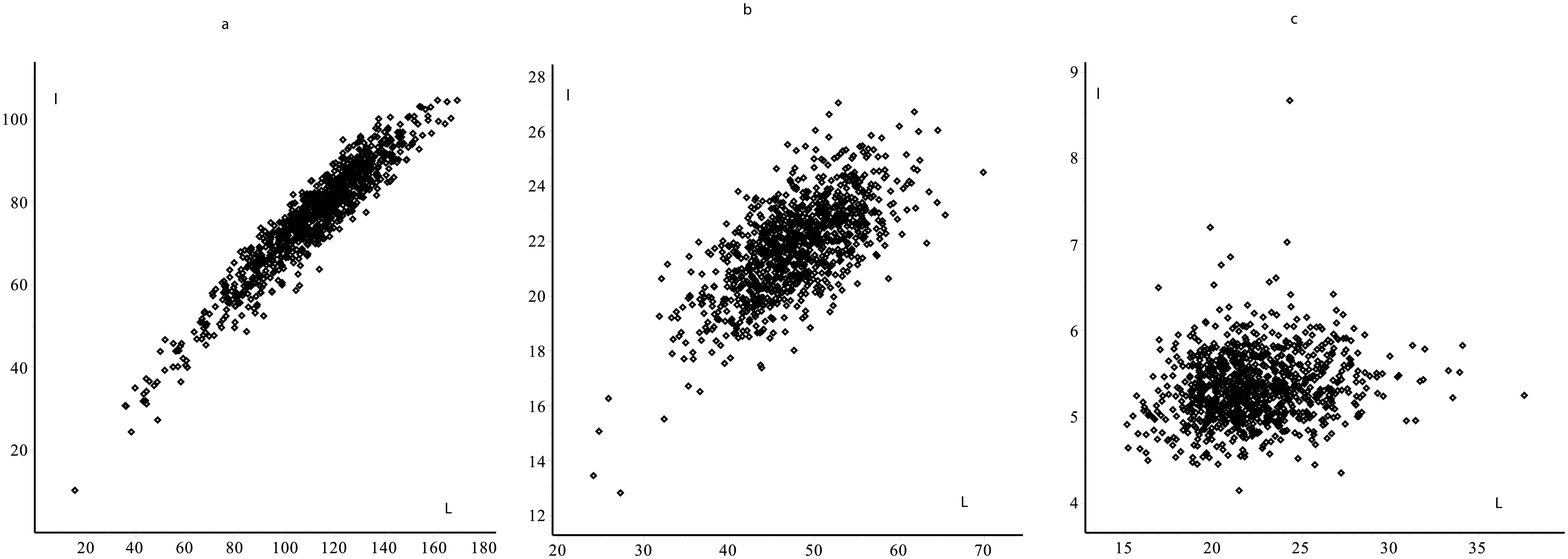}
\captionof{figure}{External versus total length. Each plot is based on 1000 coalescent realisations  with $n=1000$.  } 
\label{Fig1}

\mbox{}\\

As a byproduct, we prove the following result on the decrease of the numbers $X_k$ of branches and $Y_k$ of external branches.

\begin{theorem}
\label{mainresult3}
For $1< \alpha < 2$ as $n \to \infty$
\[ \max_{1 \le j\le \tau_n} \big| \frac{X_{\tau_n-j}}n - \frac j{\tau_n} \big| =o_P(1) \quad \text{and} \quad \max_{1\le j \le \tau_n} \big| \frac{Y_{\tau_n-j}}n - \Big(\frac j{\tau_n}\Big)^\alpha \big| =o_P(1) \ . \]
\end{theorem}

\noindent
This is illustrated in Figure \ref{Fig2}.\\\\

\psfrag{a}{$n=100$}
\psfrag{b}{$n=1000$}
\psfrag{c}{$n=10000$}
\psfrag{x}{$X$}
\psfrag{y}{$Y$}
\includegraphics[width=11.5cm]{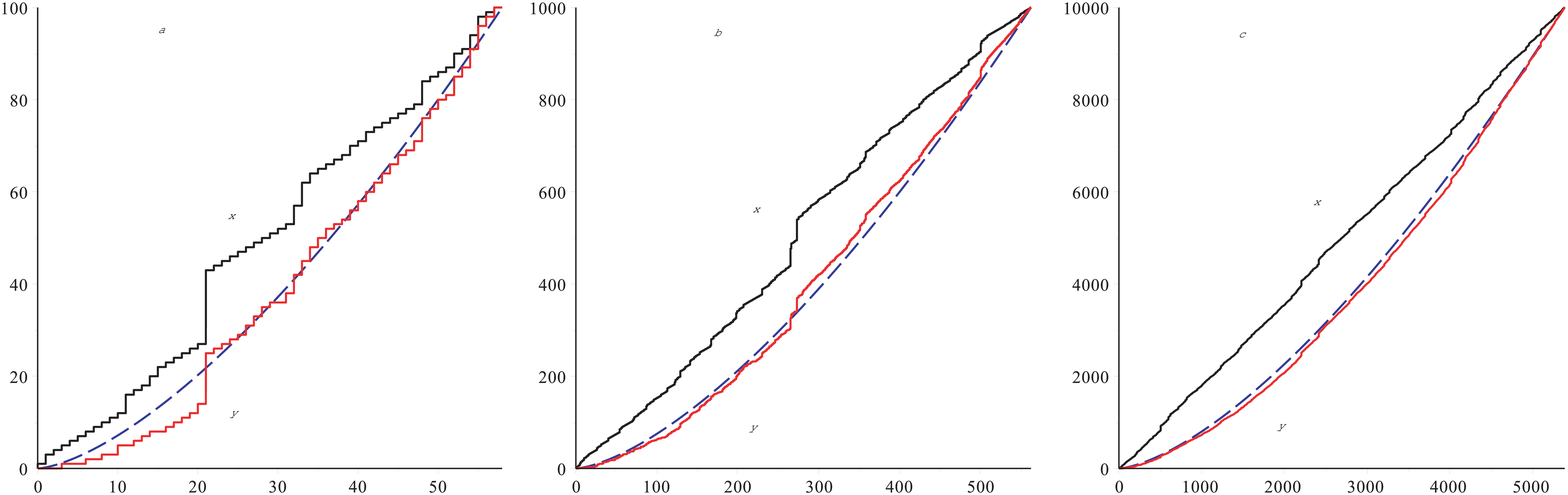}
\captionof{figure}{Simulation of $X=(X_{\tau_n-j})$ and $Y=(Y_{\tau_n-j})$ for $\alpha=1.5$ and $n=100$, 1000 and 10000. The dashed line shows the curve $n(j/\tau_n)^\alpha$. }
\label{Fig2}

\section{Heuristics and outline}
The proof of Theorem \ref{mainresult} is given in the next section. It reveals an unexpected link between $\ell_n$ and the total number of mergers $\tau_n$. In this section we explain the heuristics; this may also serve as a guide through the next section.

For $m \ge 2$, the rate at which a coalescence happens within $m$ lineages is denoted by $\lambda_m$. We thus have
\begin{align}\label{totl}
L_n \stackrel d= \sum_{k=0}^{\tau_n-1}  X_k \, W_k/\lambda_{X_k}
\end{align}
and 
\begin{align} \label{extl} \ell_n \stackrel d= \sum_{k=0}^{\tau_n-1}  Y_k \,W_k/\lambda_{X_k}
\end{align}
where $W_0, W_1, \ldots$ are i.i.d. standard exponential distributed random variables.
From Lemma 2.2 in \cite{de} we have for $m \to \infty$
\begin{align}\label{rates}
\lambda_m= \frac 1{\alpha \Gamma (\alpha)} m^\alpha + O(m^{\alpha-1}) \ . 
\end{align}
Combining \eqref{totl} and \eqref{rates} suggests, informally written,
$$L_n \approx \, \alpha\Gamma(\alpha)\sum_{k=0}^{\tau_n-1} {X_k}^{1-\alpha}.$$ 
On the other hand, a coupling of $(X_k)$ with a renewal process shows
\begin{align}\label{renas}
X_k \approx  \gamma \cdot (\tau_n-k),
\end{align}
where $\gamma = \frac 1{\alpha -1}$ 
(see Lemma \ref{uniformEstimate} for the exact statement). This might suggest that 
\begin{align}\label{asLn}
L_n \approx \alpha\Gamma(\alpha) \frac{\gamma^{1-\alpha}}{2-\alpha} \tau_n^{2-\alpha}.
\end{align}
Before discussing the (in-)validity of this asymptotics, let us turn to the external length. Since, given $X=(X_k)$, the random variables $Y_k$ result from a hypergeometric sampling (see equation \eqref{eq3} below), we obtain (see \eqref{YX})
$$\mathbf E[Y_k\mid X] = X_k \prod_{i=1}^k (1-1/X_i)$$
Replacing $Y_k$ in \eqref{extl} by this conditional expectation and again using \eqref{rates}, we are led to 
\begin{align}\label{asln1}
\ell_n \approx  \alpha\Gamma(\alpha)\sum_{k=0}^{\tau_n-1} {X_k^{1-\alpha}}   {\prod_{i=1}^k (1-1/X_i)}.
\end{align}
Plugging in \eqref{renas}, using the asymptotics $\prod_{i=1}^k \left(1-\frac{1/\gamma}{\tau_n-i}\right)\sim \,(\tau_n-k)^{-1/\gamma}$ and the cancellation of $(\tau_n-k)^{1/\gamma}$ against $(\tau_n-k)^{1-\alpha}$, we arrive at
\begin{align}\label{asln2}
\ell_n \approx  \alpha\Gamma(\alpha){\gamma^{1-\alpha}} \tau_n^{2-\alpha}
\end{align}
(see Lemma \ref{asymexp} for a precise statement).
The asymptotics of $\tau_n$ is (see Lemma \ref{blockcounting}) 
$$\tau_n \approx \frac n\gamma + n^{1/\alpha} {\rm const} \ \varsigma.$$
Using a Taylor expansion one hence obtains (see \eqref{Knhoch} for the exact statement) 
\begin{align*}\label{Knho}
\tau_n^{2-\alpha} \approx \left(\frac n\gamma\right)^{2-\alpha} + n^{1-\alpha + 1/\alpha} {\rm const'} \ \varsigma.
\end{align*}
This shows that (as $n \to \infty$) the fluctuations of $\tau_n^{2-\alpha}$ decrease with $n$ for $\alpha > \alpha_0$ and increase for $\alpha < \alpha_0$. Since (as discussed above) the fluctuations of $L_n$ are always at least of order one, the approximation  \eqref{asLn} cannot capture the fluctuations. In contrast, the approximation \eqref{asln2} turns out to be sufficiently fine also on the level of fluctuations.  The reason for this is the asymptotic flatness of the function $ (x_k) \mapsto  \sum_{k=0}^{\tau-1} {x_k^{1-\alpha}}   {\prod_{i=1}^k (1-1/x_i)}$ in the point $(x_k) = (\gamma(\tau-k))$, as revealed in the proof of Lemma \ref{Pi}.

\section{Proofs}

Theorem \ref{mainresult} is a consequence of the following two lemmas.

The first lemma on the number of mergers has been independently obtained by Delmas et al \cite{de}, Gnedin and Yakubovich \cite{gn} and Iksanov and M\"ohle \cite{ikmo2}. (Note that they use a different normalization of the limit law.) We put
\[ \gamma = \frac 1{\alpha - 1} \ . \]
\begin{lemma} \label{blockcounting}
Let $1<\alpha < 2$. Then for $n \to \infty$
\[ \frac{\tau_n- \frac n\gamma}{n^{1/\alpha}} \ \stackrel d\to\ \frac\varsigma{\gamma^{\frac1\alpha+1} \Gamma(2-\alpha)^{\frac 1\alpha}} \ . \]
\end{lemma} 
\noindent
with $\varsigma$ described in Section \ref{Intro}. In short the proof of this lemma goes along the following lines: The negative increments $X_{k-1}-X_{k}$ can  be asymptotically replaced by i.i.d. positive random variables $V_k$, $k \ge 1$, with expectation $\gamma$ and tail behavior $\mathbf P(V_k \ge v) \sim \frac 1{\Gamma(2-\alpha)} v^{-\alpha}$. From the theory of stable laws for $k \to \infty$ 
\[ \frac{X_0-X_k-\gamma k}{k^{1/\alpha}}\sim \frac{V_1+ \cdots + V_k-  \gamma k}{k^{1/\alpha} }\stackrel d \to \frac{- \varsigma}{\Gamma(2-\alpha)^{1/\alpha}} \ . \]
The choice $k=\tau_n$ gives 
\[ \frac{n-1- \gamma \tau_n}{\tau_n^{1/\alpha}} \stackrel d\to \frac{- \varsigma}{\Gamma(2-\alpha)^{1/\alpha}} \]
and thus $\tau_n\sim n/\gamma$. This entails the lemma. For details we refer to \cite{de,gn,ikmo2}.

The asymptotic expansion in the next lemma discloses the link between $\ell_n$ and $\tau_n$.

\begin{lemma} \label{asymexp}
Let $1<\alpha<2$ and $\varepsilon >0$. Then for $n \to \infty$
\[ \ell_n = \alpha\Gamma(\alpha) (\alpha-1)^{\alpha-1}\tau_n^{2-\alpha} + O_P\big( n^{\frac 2\alpha-\alpha + \varepsilon}+n^{\frac 32 - \alpha+ \varepsilon}\big) \ .\]
\end{lemma}

\noindent
Before proving this key lemma we show that Theorem \ref{mainresult} is a direct consequence of the two preceding lemmas:

\begin{proof}[Proof of Theorem \ref{mainresult}]
From Lemma \ref{blockcounting} and a Taylor expansion we obtain
\begin{align} \label{Knhoch}
\tau_n^{2-\alpha} = \Big( \frac n \gamma\Big)^{2-\alpha} + (2-\alpha) \Big( \frac n \gamma\Big)^{1-\alpha}\Big(\tau_n-\frac n \gamma\Big)  + O_P(n^{-\alpha + \frac 2\alpha}) \ . 
 \end{align}
Using $1< \alpha < 2$ and Lemma \ref{asymexp}  
\[ \ell_n = \alpha(\alpha-1)\Gamma(\alpha) n^{2-\alpha} + \alpha(2-\alpha) \Gamma(\alpha)n^{1-\alpha}\Big(\tau_n-\frac n\gamma\Big) + O_P(n^{\frac 1\alpha+1-\alpha-\varepsilon}) \]
for some $\varepsilon >0$. This implies the theorem.
\end{proof}

\noindent
Now we prepare the proof of Lemma \ref{asymexp} by deriving two other lemmas.

\begin{lemma} \label{uniformEstimate}
For any $\varepsilon >0$
\[ \max_{1\le j \le \tau_n}  j^{-\frac 1\alpha- \varepsilon }| X_{\tau_n-j} - \gamma j| = O_P(1) \ , \]
as $n \to \infty$.
\end{lemma}

\begin{proof}
First note that due to monotonicity of $(X_k)$ for any $\eta \in \mathbb N$
\begin{align*}
\max_{1\le j \le \tau_n} & j^{-\frac 1\alpha - \varepsilon  }| X_{\tau_n-j} - \gamma j|  &\le \max_{1\le j \le \tau_n-\eta}  j^{-\frac 1\alpha - \varepsilon  }| X_{\tau_n-j} - \gamma j| + X_0- X_ \eta + \gamma \eta \ .
\end{align*}
It is well-known that $X_0- X_ \eta $ is stochastically bounded uniformly in $n$, see e.g. formula (7) in \cite{ke}. Therefore it is sufficient to prove the claim for the quantity
\[   \max_{1\le j \le \tau_n-\eta_n}  j^{-\frac 1\alpha - \varepsilon  }| X_{\tau_n-j} - \gamma j| \ , \]
where $\eta_n$ is any sequence of $\mathbb N$-valued random variables, which is stochastically bounded.

This observation allows us to use a coupling device: Let $2 \le S_1<S_2< \cdots$ denote the points of the coalescent's point process $\mu$ downwards from $\infty$, as introduced in section 3 of \cite{ke}. There it is shown that one can couple $\mu$ and $X$ on a common probability space in such a way that $X_{\tau_n-j}= S_{j}$ for all $j \le \tau_n-\eta_n$, where the random variables $\eta_n $ are stochastically bounded uniformly in $n$, see Lemma 6 in \cite{ke}. Therefore it is sufficient to estimate
\[ \max_{1\le j \le \sigma(n)}  j^{-\frac 1\alpha - \varepsilon  }| S_{j} - \gamma j| \ , \]
with \[\sigma(n) = \max\{ j\ge 0: S_j \le n\} \ , \]
where we put $S_0=1$.

We shall consider
\[ M_r=\max_{\sigma(2^{r-1})< j \le \sigma(2^r)}  j^{-\frac 1\alpha - \varepsilon  }| S_{j} - \gamma j|   \]
for $r\ge 2$. For this purpose we use the coupling of $\mu$ to a stationary point process $\nu$ with points $2\le R_1 < R_2 < \ldots$, as described in section 4 of \cite{ke}. Put 
\[ \rho(n)= \max\{j\ge 0:R_j \le n\} \ , \]
also with the convention $R_0=1$. Since $R_{j+1}-R_j$ are i.i.d. random variables for $j\ge 1$ with mean $\gamma$, we have
\begin{align} \rho(n) \sim \gamma^{-1}n \text{ a.s.} \label{eq1}
\end{align}
Moreover, since $R_j=R_1+V_1+ \cdots + V_j$, where $V_1, V_2 \ldots$ are i.i.d. random variables with distribution given in formula (4) in \cite{ke}, it follows from Lemma 9 in \cite{ke}  that
\begin{align}\label{distR}
\max_{\rho(2^{r-1})\le j \le \rho(2^r)}j^{-\frac 1\alpha- \varepsilon } |R_j - \gamma j| = O(2^{- \varepsilon r/2} ) \text{ a.s.} 
\end{align}
Now equations (19) and (24) in \cite{ke} say that for 
\[ N_r= \sigma(2^r)-\sigma(2^{r-1}) \ , \quad N_r' = \rho(2^r)-\rho(2^{r-1})  \]
with $r \ge 2$ we have
\[ |N_r-N_r'| \le D_r \]
where the random variables 
\[D_r = \max_{0 \le i \le N_r\wedge N_r'} \big|S_{\sigma(2^r)-i} - R_{\rho(2^r)-i} \big|\]
 fulfil
\begin{align*}
 \mathbf P( D_r>t) \le ct^{1-\alpha}, \quad t > 0, 
\end{align*}
for some $c>0$. Hence it follows by the Borel-Cantelli lemma that
\begin{align}
D_r = O(r^{\gamma + \varepsilon}) \text{ a.s. , }
\label{eq2}
\end{align}
 and therefore from
$ |\sigma(2^r)-\rho(2^r)| \le D_r+ \cdots + D_1+1 $ 
\begin{align}
 |\sigma(2^r)-\rho(2^r)| =O(r^{\gamma +1 +\varepsilon}) \text{ a.s.}
 \label{F}
\end{align}
for any $\varepsilon >0$. Equation \eqref{eq1} implies
\begin{align} \sigma(2^r) \sim \gamma^{-1} 2^{r} \text{ a.s.} 
\label{eqA}
\end{align}
for $r \to \infty$. 
By the triangle inequality (and reparametrising by $j= \sigma(2^r)-i)$)
\begin{align*}
\sigma(2^{r-1})^{\frac 1\alpha+\varepsilon} M_r &\le \max_{\sigma(2^{r-1})< j \le \sigma(2^r)} | S_{j} - \gamma j| \\ &\le \max_{0 \le i < N_r} \big|S_{\sigma(2^r)-i} - R_{(\rho(2^r)-i)\vee \rho(2^{r-1})} \big|\\
&\mbox{} \qquad + \max_{\rho(2^{r-1})\le j \le \rho(2^r)} |R_j - \gamma j| \\
&\mbox{} \qquad + \gamma \max_{0 \le i < N_r} \big|(\rho(2^r)-i)\vee \rho(2^{r-1})- (\sigma(2^r)-i) \big| \ .
\end{align*}
In case $N_r'< N_r$ the first maximum on the right-hand side is not attained for $i > N_r'$, because then $|S_{\sigma(2^r)-i} - R_{(\rho(2^r)-i)\vee \rho(2^{r-1})} | =  S_{\sigma(2^r)-i}- R_{\rho(2^{r-1})}$ is decreasing for $i \ge N_r'$. Also the third maximum is attained either for $i=0$ or $i=N_r$. This implies
\begin{align*}
\sigma(2^{r-1})^{\frac 1\alpha+\varepsilon} M_r &\le  \max_{0 \le i \le (N_r-1)\wedge N_r'} \big|S_{\sigma(2^r)-i} - R_{\rho(2^r)-i} \big| \\
&\mbox{} \qquad + \gamma \big| \sigma(2^r)-\rho(2^r) \big|+ \gamma \big| \sigma(2^{r-1})-\rho(2^{r-1}) \big|\\
&\mbox{} \qquad + \max_{\rho(2^{r-1})\le j \le \rho(2^r)} |R_j - \gamma j|
\end{align*}
From \eqref{eq2} and \eqref{eqA} for $\varepsilon$ sufficiently small
\begin{align*}M_r &=  \sigma(2^{r-1})^{-\frac 1\alpha - \varepsilon } \max_{\rho(2^{r-1})\le j \le \rho(2^r)} |R_j - \gamma j|+ O(2^{- \frac r\alpha}  r^{1+\gamma + \varepsilon}) \\
&\le 2  \max_{\rho(2^{r-1})\le j \le \rho(2^r)}j^{-\frac 1\alpha - \varepsilon } |R_j - \gamma j|+ O(2^{-\varepsilon r})
\end{align*}
Using \eqref{distR} we infer
\[ M_r = O(2^{-\varepsilon r/2})  \text{ a.s.} \]
Now for $n \in \mathbb N$, choosing $s\in \mathbb N$ such that $2^{s-1} < n \le 2^s$, since $\sigma(n) \le n$ 
\begin{align*}
\max_{1\le j \le \sigma(n)}  j^{-\frac 1\alpha - \varepsilon  }| S_{j} - \gamma j| \le \sum_{r=1}^s M_r = O\big( \sum_{r=1}^s 2^{-\varepsilon r /2}\big) = O(1) \text{ a.s.} 
\end{align*}
This gives the claim.
\end{proof}
Now observe the following: Not only $X=(X_0, \ldots, X_{\tau_n})$ is a Markov chain, but given $X$ also the sequence $Y=(Y_0, \ldots, Y_{\tau_n})$. More precisely let $U_k$ be the number of particles collapsing to one particle in the $k$-th merging event,
thus
\begin{align}\label{XU}
X_k = X_{k-1} - U_k+1 \ .
\end{align} 
For $k>\tau_n$ we have $U_k=1$, since by definition $1=X_{\tau_n}= X_{\tau_{n+1}}= \cdots$. The  numbers of external branches, $n=Y_0>\cdots> Y_{\tau_n}=0$, follow the recursion
\begin{align}
Y_{k} = Y_{k-1} -H_k 
\label{eq3}
\end{align}
with
\begin{align*}
(H_k \mid X, Y_0, \ldots, Y_{k-1} ) \stackrel d= \text{Hyp}(X_{k-1},Y_{k-1},U_k) \ , 
\end{align*}
where $\text{Hyp}(N,M,\nu)$ denotes the hypergeometric distribution with parameters $N,M,\nu$. (These formulas also hold for $k>\tau_n$.) Thus
\[ \mathbf E[Y_k \mid X, Y_0, \ldots, Y_{k-1}] = Y_{k-1} - U_k \frac {Y_{k-1}}{X_{k-1}}= \frac{X_{k}-1}{X_{k-1}} Y_{k-1} \]
and
\begin{align}
 \frac{\mathbf E[Y_k \mid X]}{X_k}= \frac{X_{k}-1}{X_kX_{k-1}} \mathbf E[Y_{k-1}\mid X] = \frac{\mathbf E[Y_{k-1} \mid X]}{X_{k-1}} \Big(1-\frac{1}{X_{k}}\Big) \ .
 \label{eq4}
 \end{align}
Iterating this formula and observing that $\mathbf E[Y_0 \mid X] =Y_0=X_0$ we obtain
\begin{align}\label{YX}
\mathbf E[Y_k \mid X] = X_k\Pi_0^k \ ,
\end{align}
where
\[ \Pi_j^k : = \prod_{i=j+1}^k \Big( 1- \frac 1{X_i}\Big) \ , \quad 0 \le j \le k \ ,\]
in particular $\Pi_k^k=1$.

Next we study the auxiliary quantity $\Pi_j^k$.
\begin{lemma} \label{Pi}
Let $\varepsilon >0$. Then for $0 \le j \le k < \tau_n$ as $n \to \infty$
\begin{align}\label{Pi1}
 \Pi_j^k = \Big(\frac{\tau_n-k}{\tau_n-j}\Big)^{\frac 1\gamma} \Big(1+ \sum_{i=j+1}^k \frac{X_{i}- \gamma(\tau_n -i)}{\gamma^2(\tau_n-i)^{2}} + O_P \big( (\tau_n-k)^{  \frac 2\alpha-2+\varepsilon}\big) \Big) 
 \end{align}
and
\begin{align}\label{Pi2}
\Pi_j^k = \Big(\frac{\tau_n-k}{\tau_n-j}\Big)^{\frac 1\gamma} \Big(1 + O_P \big( (\tau_n-k)^{  \frac 1\alpha-1+\varepsilon}\big) \Big)\ , 
\end{align}
where the $O_P(\cdot)$ terms hold uniformly for all $0 \le j \le k < \tau_n$.
\end{lemma} 

\begin{proof}
Using a Taylor expansion of $\log(1+h)$ we obtain that on the event $\{k< \tau_n\}$
\begin{align*}
\Pi_j^k= \exp\Big( -\sum_{i=j+1}^k \frac 1{X_i}  + O\big( \sum_{i=j+1}^k \frac 1{X_i^2} \big) \Big) 
\end{align*}
with the error term uniformly bounded in $n$ and $0 \le j \le k < \tau_n$. 
In view of a Taylor expansion together with Lemma \ref{uniformEstimate} and the fact that $\sum_{m=r}^{s-1} \frac 1m = \log \frac sr + O(\frac 1r)$ we obtain
\begin{align*}\sum_{i=j+1}^k \frac 1{X_i}  &= \sum_{i=j+1}^k \Big( \frac 1{\gamma (\tau_n-i)}- \frac{X_{i}- \gamma(\tau_n -i)}{ \gamma^2(\tau_n-i)^2} + O_P \Big(\frac{(X_{i}- \gamma(\tau_n -i))^2}{(\tau_n-i)^3}\Big)\Big) \\
&= \frac 1\gamma\log \frac{\tau_n-j}{\tau_n-k} -  \sum_{i=j+1}^k \frac{X_{i}- \gamma(\tau_n -i)}{\gamma^2(\tau_n-i)^{2}}  +
  O_P \big( (\tau_n-k)^{  \frac 2\alpha-2+\varepsilon}\big) \ ,
\end{align*}
where the $O_P(\cdot)$ holds uniformly for all $0 \le j \le k < \tau_n$. Also from Lemma \ref{uniformEstimate}
\[ \sum_{i=j+1}^k \frac 1{X_i^2} = O_P \Big( \sum_{i=j+1}^k \frac 1{(\tau_n-i)^2}\Big) = O_{P} \big((\tau_n-k)^{-1}\big) \ , \]
therefore for $0 \le j \le k<\tau_n$, since $1<\alpha < 2$,
\begin{align*}
\Pi_j^k  = \Big(\frac{\tau_n-k}{\tau_n-j}\Big)^{\frac 1\gamma} \exp\Big( \sum_{i=j+1}^k \frac{X_{i}- \gamma(\tau_n -i)}{\gamma^2(\tau_n-i)^{2}} + O_P \big( (\tau_n-k)^{  \frac 2\alpha-2+\varepsilon}\big) \Big) \ .
\end{align*}
This shows \eqref{Pi1}. Finally from Lemma \ref{uniformEstimate}, since $\alpha >1$,
\[ \sum_{i=j+1}^k \frac{X_{i}- \gamma(\tau_n -i)}{(\tau_n-i)^{2}} = O_P \Big(\sum_{i=j+1}^k (\tau_n-i)^{\frac 1\alpha -2 + \varepsilon}\Big)= O_P\big((\tau_n-k)^{\frac 1\alpha -1+ \varepsilon}\big)\ , \]
hence also \eqref{Pi2} follows.
\end{proof}

\begin{proof}[Proof of Lemma \ref{asymexp}]
We shall approximate $\ell_n$ step by step. First let 
\[ \mathfrak l_n= \sum_{k=0}^{\tau_n-1} \frac{\mathbf E[Y_k \mid X]}{X_k^\alpha}= \sum_{k=0}^{\tau_n-1} X_k^{1-\alpha}\Pi_0^k  \ . \]
To obtain an approximation to this expression we use the following expansion, following from Lemma \ref{uniformEstimate}: Uniformly for $k < \tau_n$ 
\begin{align*}
X_k^{1-\alpha} = (\gamma(\tau_n-k))^{1-\alpha} &+ (1-\alpha)(\gamma(\tau_n-k))^{-\alpha}(X_k-\gamma(\tau_n-k) )\\
&\mbox{}\quad + O_P\big((\tau_n-k)^{-\alpha-1}(X_k-\gamma (\tau_n-k))^2\big) 
\end{align*}
and thus
\begin{align*}
X_k^{1-\alpha} = (\gamma(\tau_n-k))^{1-\alpha}\Big(1- \frac {X_k-\gamma (\tau_n-k)}{\gamma^2 (\tau_n-k)} + O_P\big( (\tau_n-k)^{\frac 2\alpha-2 + \varepsilon}\big) \Big)\ . 
\end{align*}
Together with Lemma \ref{Pi} we obtain (taking the first summand $n^{1-\alpha}$ apart)
\begin{align*}
\mathfrak l_n= n^{1-\alpha}+(\gamma \tau_n)^{1-\alpha} \sum_{k=1}^{\tau_n-1} \Big(1 &+\sum_{i=1}^k \frac{X_{i}- \gamma(\tau_n -i)}{\gamma^2(\tau_n-i)^{2}}\\ &\mbox{}  - \frac {X_k-\gamma (\tau_n-k)}{\gamma^2 (\tau_n-k)} + O_P\big( (\tau_n-k)^{\frac 2\alpha-2 + \varepsilon}\big) \Big)\ .
\end{align*}
Now
\[ \sum_{k=1}^{\tau_n-1} \sum_{i=1}^k \frac{X_i-\gamma(\tau_n-i)}{\gamma^2(\tau_n-i)^2} = \sum_{i=1}^{\tau_n-1}\frac{X_i-\gamma(\tau_n-i)}{\gamma^2(\tau_n-i)} \ ,\]
such that (surprisingly) matters simplify to
\[ \mathfrak l_n =  n^{1-\alpha} +(\gamma \tau_n)^{1-\alpha} \big(\tau_n + O_P( \tau_n^{\frac 2\alpha-1 + \varepsilon})\big) \]
or, using Lemma \ref{blockcounting},
\begin{align} 
\mathfrak l_n =  \gamma^{1-\alpha}\tau_n^{2-\alpha} + O_P( n^{\frac 2\alpha-\alpha + \varepsilon}) \ . \label{lambda}
\end{align}

Next we consider
\[ \mathfrak l_n'= \sum_{k=0}^{\tau_n-1} \frac{Y_k}{X_k^\alpha} \ . \]
Putting $G_k= Y_k-\mathbf E[Y_k\mid X]$, we obtain, since $G_0=0$,
\[ \mathfrak l_n'-\mathfrak l_n = \sum_{k=1}^{\tau_n-1} \frac{G_k}{X_k^\alpha}   \ . \]
Now from \eqref{eq3} and \eqref{XU}
\[ Y_k = Y_{k-1}\frac{X_k-1}{X_{k-1}} - \tilde H_k \ , \quad \text{with } \tilde H_k=H_k- U_k\frac {Y_{k-1}}{X_{k-1}} \ ,\]
and thus from \eqref{eq4}
\[  G_k = G_{k-1}\frac{X_k-1}{X_{k-1}} - \tilde H_k\]
or
\[ \frac{G_k}{X_k\Pi_0^k} = \frac{G_{k-1}}{X_{k-1}\Pi_0^{k-1}} -  \frac{\tilde H_k}{X_k\Pi_0^k} \ . \]
Iterating this relation we obtain because of $G_0=0$
\begin{align} G_k = -X_k\Pi_0^k\sum_{i=1}^k \frac{\tilde H_i}{X_i\Pi_0^i} \ , 
\label{G}
\end{align}
thus
\[ \mathfrak l_n-\mathfrak l_n' = \sum_{k=1}^{\tau_n-1} \frac{\Pi_0^k}{X_k^{\alpha -1}} \sum_{i=1}^k \frac{\tilde H_i}{X_i\Pi_0^i} = \sum_{i=1}^{\tau_n-1}\frac{\tilde H_i}{X_i} \sum_{k=i}^{\tau_n-1} \frac{\Pi_i^k}{X_k^{\alpha -1}} \ . \]
Now note that given $ X,Y_0, \ldots,Y_{i-1}$ the random variable $\tilde H_i$ has a hypergeometric distribution centered at expectation. In fact $(\tilde H_i)$ forms a sequence of martingale differences w.r. to the filtration generated by $(X, Y_0,Y_1,\ldots, Y_i)$, $i \ge 0$,  therefore $\tilde H_1, \tilde H_2, \ldots$ are uncorrelated given $X$. Also from the formula for the variance of a hypergeometric distribution 
\begin{align}\mathbf E[\tilde H_i^2 \mid X, Y_0, \ldots, Y_{i-1}] \le U_i \frac{Y_{i-1}}{X_{i-1}} \ . 
\label{H}
\end{align}
This entails
\begin{align*}
\mathbf E[(\mathfrak l_n'-\mathfrak l_n)^2 \mid X] &\le  \sum_{i=1}^{\tau_n-1}\frac 1{X_i^2}\frac{U_i\mathbf E[Y_{i-1}\mid X]}{X_{i-1}} \Big(\sum_{k=i}^{\tau_n-1} \frac{\Pi_i^k}{X_k^{\alpha -1}}\Big)^2 \\
&= \sum_{i=1}^{\tau_n-1}U_i\frac {\Pi_0^{i-1}}{X_i^2}\Big(\sum_{k=i}^{\tau_n-1} \frac{\Pi_i^k}{X_k^{\alpha -1}}\Big)^2\ .
\end{align*}
Using a Taylor expansion similar as in the proof of Lemma  \ref{Pi} and applying Lemma \ref{uniformEstimate} we see that  $X_i^{-1}$ is of order $(\tau_n-i)^{-1}$ uniformly in $0\le i < \tau_n$. This together with Lemma \ref{Pi} implies
\begin{align*} 
\mathbf E[&(\mathfrak l_n'-\mathfrak l_n)^2 \mid X] \\&= O_P \Big( \sum_{i=1}^{\tau_n-1} U_i \Big(\frac{\tau_n-i+1}{\tau_n}\Big)^{\frac 1\gamma}\frac 1{(\tau_n-i)^2} \Big(\sum_{k=i}^{\tau_n-1}\Big( \frac{\tau_n-k}{\tau_n-i}\Big)^{\frac 1\gamma}\frac 1{(\tau_n-k)^{\alpha-1}} \Big)^2\Big) \ .
\end{align*}
Since $\frac 1\gamma = \alpha-1$, this boils down to
\[\mathbf E[(\mathfrak l_n'-\mathfrak l_n)^2 \mid X] = O_P \Big( \tau_n^{1-\alpha} \sum_{i=1}^{\tau_n-1}  U_i( \tau_{n}-i)^{1-\alpha} \Big)
\]
Again from Lemma \ref{uniformEstimate} $X_{i-1}$ is of order $\tau_n-i$ uniformly in $0\le i < \tau_n$. Also $U_i = X_{i-1}-X_i+1 \le 2(X_{i-1}-X_i)$, therefore
\[\mathbf E[(\mathfrak l_n'-\mathfrak l_n)^2 \mid X] = O_P \Big(\tau_n^{1-\alpha} \sum_{i=1}^{\tau_n-1} X_{i-1}^{1-\alpha}(X_{i-1}-X_i) \Big)\ . \]
Moreover 
\[ \sum_{i=1}^{\tau_n-1} X_{i-1}^{1-\alpha}(X_{i-1}-X_i) \le \int_{X_{\tau_n}}^{X_0} x^{1-\alpha} \, dx \le \frac 1{2-\alpha} X_0^{2-\alpha} \ . \]
Since $X_0=n$, we end up with
\[\mathbf E[(\mathfrak l_n'-\mathfrak l_n)^2 \mid X] = O_P (n^{3-2\alpha}) \ . \]
Chebychev's inequality and dominated convergens now imply for any $\varepsilon >0$
\begin{align*}
\mathbf P( |\mathfrak l_n'-\mathfrak l_n| > n^{\frac 32 - \alpha+ \varepsilon}) &= \mathbf E\big[ \mathbf P( |\mathfrak l_n'-\mathfrak l_n| > n^{\frac 32 - \alpha+ \varepsilon})\mid X\big] \\
&\le \mathbf E\big[ 1 \wedge (n^{-3+ 2\alpha - 2\varepsilon} \mathbf E[(\mathfrak l_n'-\mathfrak l_n)^2 \mid X] )\big] = o(1)
\end{align*}
thus $ \mathfrak l_n'-\mathfrak l_n = O_P(n^{\frac 32 - \alpha+ \varepsilon})$. Combined with \eqref{lambda} we obtain
\begin{align}
\mathfrak l_n' =\gamma^{1-\alpha}\tau_n^{2-\alpha} + O_P( n^{\frac 2\alpha-\alpha + \varepsilon}+n^{\frac 32 - \alpha+ \varepsilon}) \label{lambda2}
\end{align}
for all $\varepsilon >0$.

Next we prepare to incorporate the waiting times between mergers into our expression. Let $W_1, W_2, \ldots$ be independent exponential random variables with mean 1, also independent from $(X,Y)$. Let
\[ \mathfrak l_n'' = \sum_{k=0}^{\tau_n-1} \frac {Y_k}{X_k^\alpha} W_k \ .\]
Then
\[\mathbf E[(\mathfrak l_n''- \mathfrak l_n')^2 \mid X,Y] = \sum_{k=0}^{\tau_n-1} \frac{Y_k^2}{X_k^{2\alpha}} \le \sum_{k=0}^{\tau_n-1} \frac{Y_k}{X_k^{2\alpha-1}} \, \]
and therefore
\[\mathbf E[(\mathfrak l_n''- \mathfrak l_n')^2 \mid X]  \le \sum_{k=0}^{\tau_n-1} \frac{\mathbf E[Y_k \mid X]}{X_k^{2\alpha-1}}= \sum_{k=0}^{\tau_n-1} X_k^{2-2\alpha}\Pi_0^k \ . \]
From Lemmas \ref{uniformEstimate}, \ref{Pi} and $\alpha <2$ it follows that
\[\mathbf E[(\mathfrak l_n''- \mathfrak l_n')^2 \mid X]  = O_P(\tau_n^{1-\alpha}\sum_{k=0}^{\tau_n-1} (\tau_n-k)^{1-\alpha} ) = O_P(n^{3-2\alpha}) \ .\]
As above we may conclude $ \mathfrak l_n''-\mathfrak l_n' = O_P(n^{\frac 32 - \alpha+ \varepsilon})$ and consequently from \eqref{lambda2}
\begin{align}
\mathfrak l_n'' =\gamma^{1-\alpha}\tau_n^{2-\alpha} + O_P( n^{\frac 2\alpha-\alpha + \varepsilon}+n^{\frac 32 - \alpha+ \varepsilon}) \label{lambda3} \ . 
\end{align}

In a last step we add the coalescence rate $\lambda_M$ into our formulas. 
From Lemmas \ref{uniformEstimate} and \ref{Pi}
\begin{align}\label{last}
\mathbf E\Big[ \frac {Y_k}{X_k^{\alpha+1}} W_k \mid X \Big] = \frac {\Pi_0^k}{X_k^\alpha} =O_P(\tau_n^{1-\alpha}(\tau_n-k)^{-1}) \ ,
\end{align}
where the $O_P(\cdot)$ holds uniformly for all $0\le k\le \tau_n$. Therefore, taking the sum in \eqref{last} and using Markov's inequality, 
\[ \sum_{k=0}^{\tau_n-1} \frac{Y_k}{X_k^{\alpha+1}} W_k = O_P\Big(\tau_n^{1-\alpha} \sum_{k=0}^{\tau_n-1} (\tau_n-k)^{-1}\Big)= O_P (\tau_n^{1-\alpha} \log \tau_n) \ . \]
Combined with \eqref{rates} and \eqref{lambda3} this gives
\[ \sum_{k=0}^{\tau_n-1} \frac {Y_k}{\lambda_{X_k}} \, W_k = \alpha\Gamma(\alpha) \gamma^{1-\alpha}\tau_n^{2-\alpha} + O_P( n^{\frac 2\alpha-\alpha + \varepsilon}+n^{\frac 32 - \alpha+ \varepsilon}) \ .  \]
Now the left-hand quantity is equal to $\ell_n$ in distribution. This proves Lemma~\ref{asymexp}.
\end{proof}

\begin{proof}[Proof of Theorem \ref{mainresult3}]
The first statement is a direct consequence of Lemma \ref{uniformEstimate} with $\varepsilon=1-1/\alpha$ and Lemma \ref{blockcounting}. For the proof of the second one note that in view of monotonicity of $(Y_k)$ it is sufficient to show that for $0<c < 1$ and $k_n \sim c \tau_n$
\[ \frac{ Y_{k_n}}{n} \to (1-c)^{\alpha} \]
in probability. From \eqref{YX},  \eqref{Pi2} and the first statement of the theorem
\[ \frac{\mathbf E[ Y_{k_n} \mid X]}n = \frac{X_{k_n}}n \cdot \Pi_0^{k_n} \to (1-c) \cdot (1-c)^{\alpha -1}\]
in probability. Thus it remains to prove that $G_{k_n}= Y_{k_n}- \mathbf E[Y_{k_n} \mid X]$ is of order $o_P(n)$. Now from \eqref{G}, \eqref{H} and the explanation given there,
\[ \mathbf E[G_k^2 \mid X]  \le (X_k\Pi_0^k)^2 \sum_{i=1}^k \frac 1{(X_i\Pi_0^i)^2} U_i \Pi_0^{i-1} \ .  \]
As in the proof of Lemma \ref{asymexp} (see the arguments following the inequality \eqref{H}) we mutually replace $X_i$ and $\tau_n-i$ as well as $\Pi_0^i$ and $(\tau_n-i)^{1/\gamma}/\tau_n^{1/\gamma}$ to obtain uniformly in $k$
\begin{align*} \mathbf E[G_k^2 \mid X] &= O_P \Big(\frac{ (\tau_n-k)^{2\alpha}}{\tau_n^{\alpha -1}}\sum_{i=1}^k (\tau_n-i)^{-\alpha -1} U_i \Big)\\
&=  O_P \Big(\frac{ (\tau_n-k)^{2\alpha}}{\tau_n^{\alpha -1}}\sum_{i=1}^k X_i^{-\alpha-1}(X_{i-1}-X_i)\Big)
\end{align*}
and uniformly in $k$
\[ \mathbf E[G_k^2 \mid X] =O_P\Big( \frac{ (\tau_n-k)^{2\alpha}}{\tau_n^{\alpha -1}} X_k^{-\alpha}\Big) =O_P(\tau_n)= O_P(n) \ . \]
This implies that $G_{k_n}$ is of order $O_P(\sqrt n)$, which gives our claim.
\end{proof}

\end{document}